\documentclass{proc-l}

\newtheorem{theorem}{Theorem}[section]
\newtheorem{lemma}[theorem]{Lemma}

\theoremstyle{definition}

\theoremstyle{remark}

\numberwithin{equation}{section}

\usepackage[utf8]{inputenc}
\usepackage[russian,english]{babel}
\usepackage{amssymb,amscd,amsthm, verbatim,amsmath,color,fancyhdr, mathrsfs}
\usepackage{mathtools, nccmath}
\begin{document}

\title{On the Baire class of n-Dimensional Boundary Functions}


\author[C. P. Wilson]{Connor Paul Wilson}
\address{530 Church Street
Ann Arbor, MI 48109}
\email{dpoae@umich.edu}


\date{}


\begin{abstract}
We show an extention of a theorem of Kaczynski to boundary functions in n-dimensional space. Let $H$ denote the upper half-plane, and let $X$ denote its frontier, the $x$-axis. Suppose that $f$ is a function mapping $H$ into some metric space $Y.$ If $E$ is any subset of $X,$ we will say that a function $\varphi: E \rightarrow Y$ is a boundary function for $f$ if and only if for each $x\in E$ there exists an arc $\gamma$ at $x$ such that $\lim_{z\rightarrow x \atop z\in\gamma} f(z) = \varphi(x)$
\end{abstract}

\maketitle

\section{Introduction}
\subsection{Preliminaries and Notation}

Let $H$ denote the upper half-plane, and let $X$ denote its frontier, the $x$-axis. If $x\in X,$ then by an arc at $x$ we mean a a simple arc $\gamma$ with one endpoint at $x$ such that $\gamma - \{x\} \subseteq H.$ Suppose that $f$ is a function mapping $H$ into some metric space $Y.$ If $E$ is any subset of $X,$ we will say that a function $\varphi: E \rightarrow Y$ is a boundary function for $f$ if and only if for each $x\in E$ there exists an arc $\gamma$ at $x$ such that
$$
\lim_{z\rightarrow x \atop z\in\gamma} f(z) = \varphi(x)
$$
We will also define Baire classes as Kaczynski does in \cite{Kacz}, such that a function $f: M \rightarrow Y$ is said to be of Baire class $O(M, Y)$ if and only if it is continuous; if $\xi$ is an ordinal number greater than or equal to $1,$ then f is said to be of Baire class $\xi(M, Y)$ if and only if there exists a sequence of functions $\left\{f_{n}\right\}_{n=1}^{\infty}$ mapping M into $Y, f_{n}$ being of Baire class $\eta_{n}(M, Y)$ for some $\eta_{n}<\xi$, such that $f_{n} \rightarrow f$ pointwise.

\section{Boundary functions for discontinuous functions}

\begin{theorem}
Let $Y$ be a separable arc-wise connected metric space, with $f: H \rightarrow Y$ a function of Baire class $\xi(H, Y)$ where $\xi \geq 1,$ $E$ a subset of $X$, and $\varphi: E \rightarrow Y$ a boundary function for $f$. Therefore $\varphi$ is of Baire class $\xi + 1(E, Y).$
\end{theorem}

\begin{proof}
Let $U$ be an open subset of $Y$ such that $V =  Y-\operatorname{clos}(U)$. Set $A = \varphi^{-1}(U),\ B = \varphi^{-1}(V),\ C = A\cup B.$ Notice that we clearly have an empty intersection between $A$ and $B$.
$\forall x\in C$, choose an arc $\gamma_{x}$ at $x$ such that:

$$
\lim_{z\rightarrow x \atop z\in\gamma_{x}} f(z) = \varphi(x)
$$
with
$$
\gamma_{x} \subseteq\{z: \mid z - x \mid \leq 1\}
$$
where
$$
\begin{cases}
\gamma_{x} -\{x\}\subseteq f^{-1}(U)& \text{ if } x\in A \\ 
\gamma_{x} -\{x\}\subseteq f^{-1}(V) & \text{ if } x\in B 
\end{cases}
$$
Note once again the empty intersection $\gamma_{x} \cap \gamma_{y} = \emptyset$ for $x\in A\ \wedge\ y\in B$

Let us define the terminology $\gamma_{x}$ \textit{meets} $\gamma_{y}$ in $\operatorname{clos}(H_{n})$ provided that the two arcs have respective subarcs, $\gamma_{x}\prime$ and $\gamma_{y}\prime$ with $x\in \gamma_{x}\prime \subseteq \operatorname{clos}(H_{n})$ and $x\in \gamma_{y}\prime \subseteq \operatorname{clos}(H_{n})$, with $\gamma_{x}\prime \cap \gamma_{y}\prime \neq \varnothing$

Let:
$$
L_{a}:={x\in A : \forall n\exists y, \text{ such that } y\in C,\ y\neq x, \ \gamma_{y} \text{ meets } \gamma_{x} \text{ in }\operatorname{clos}(H_{n})}
$$
$$
L_{b}:={x\in B : \forall n\exists y, \text{ such that } y\in C,\ y\neq x,\  \gamma_{y} \text{ meets } \gamma_{x} \text{ in }\operatorname{clos}(H_{n})}
$$
$$
M_{a}:={x\in A : \exists n, \gamma_{x} \text{ meets no } \gamma_{y} \text{ in }\operatorname{clos}(H_{n})}
$$
$$
M_{b}:={x\in B : \exists n, \gamma_{x} \text{ meets no } \gamma_{y} \text{ in }\operatorname{clos}(H_{n})}
$$
$$
L = L_{a} \cup L_{b}
$$
$$
M = M_{a} \cup M_{b}
$$
$L_{a}, L_{b}, M_{a}, $ and $M_{b}$ are notably pairwise disjoint, and $A = L_{a} \cup M_{a},$ $B = L_{b} \cup M_{b}.$

Let $n(x) \in \mathbb{Z}^{+}$ for each $x \in M$ such that $\gamma_{x}$ meets no $\gamma_{y}$ in $\operatorname{clos}(H_{n(x)})$ such that $x \neq y,$ where $n \geq n(x)$ gives the obvious no meeting case in $\operatorname{clos}(H_{n})$. Moreover, take 
$$
K_{n} := \{x\in C : \gamma_{x} \text{ meets } X_{n} \wedge \text{ if } x\in M, n \geq n(x)\}
$$
It is clear that we have for every $n, K_{n} \subseteq K_{n+1}$, as well as $C = \bigcup_{n=1}^{\infty}K_{n}.$ It follows by the work of Kaczynski\cite{Kacz} and the following lemma that we have Theorem $2.1$.

\begin{lemma}
Let Y be a separable arc-wise connected metric space, E any metric space, and $\varphi: E\rightarrow Y$ a function such that for every open set $U \subseteq Y$ there exists a set $T \in P^{\xi + 1}(E)$ where $\varphi^{-1}(U)\subseteq T\subseteq \varphi^{-1}(\operatorname{clos}(U))$. Thus for $\xi \geq 2,$ $\varphi$ is of Baire class $\xi(E, Y).$
\end{lemma}

\begin{proof}
Let $\mathcal{B}$ be a countable base for $Y,$ and suppose $W$ is some open subset of $Y.$ Let:
$$
\mathcal{A}(W) = \{U\in \mathcal{B} : \operatorname{clos}(U)\subseteq W\}
$$
By Kaczynski, we take:
$$
W = \bigcup_{U\in \mathcal{A}(W)} U = \bigcup_{U\in \mathcal{A}(W)}\operatorname{clos}(U).
$$
$\forall U\in \mathcal{B},$ let $T(U) \in P^{\xi + 1}(E)$ be chosen so that $\varphi^{-1}(U)\subseteq T(U)\subseteq \varphi^{-1}(\operatorname{clos}(U)).$ Thus we have:
$$
\begin{aligned}
\varphi^{-1}(W) &= \bigcup_{U\in \mathcal{A}(W)}\varphi^{-1}(U) \subseteq \bigcup_{U\in \mathcal{A}(W)}T(U)  \\
&\subseteq \bigcup_{U\in \mathcal{A}(W)}\varphi^{-1}(\operatorname{clos}(U)) = \varphi^{-1}(W).
\end{aligned}
$$
Therefore $\varphi^{-1}(W) = \bigcup_{U\in \mathcal{A}(W)}T(U)$, and given $P^{\xi + 1}(E)$ is closed under countable unions, we have $\varphi^{-1}(W)\in P^{\xi + 1}(E),$ and $\varphi$ is of Baire class $\xi(E, Y)$
\end{proof}
And following from above we therefore have:
$$
\varphi^{-1}(U) = A \subseteq T \cap E \subseteq E - B = E -\varphi^{-1}(V) = \varphi^{-1}(\operatorname{clos}(U))
$$
for some $T \in P^{\xi + 2}(X)$, and we know $T\cap E \in P^{\xi + 2}(E)$ which by the above lemma gives us $\varphi$ of Baire class $\xi + 1(E, Y).$
\end{proof}

\section{Sets of curvilinear convergence for $\mathbb{R}^{3}$}

Let $f: H \rightarrow Y$ of Baire class $\xi(H, Y)$, and $\varphi: E \rightarrow Y$ a boundary function of $f.$ Let us define some function to analyse the properties of $M_{a},$ following Kaczynski. Take $\pi: \mathbb{R}^{3}\rightarrow\mathbb{R}^{2}$ such that 
$$
\pi(\langle x, y, z\rangle) = \left \| \langle x, y\rangle \right \|_{2}.
$$
If $\left \| \langle x, y\rangle \right \|_{2} \in M \cap K_{n},$ then let us define $p_{n}(\left \| \langle x, y\rangle \right \|_{2})$ as the first point of $X_{n}$ reached along $\gamma_{x}$ starting from $x.$ It is clear thus that by Kaczynski, the function $\pi(p_{n}(\left \| \langle x, y\rangle \right \|_{2})$ is strictly increasing on $M \cap K_{n}.$ Thus by the above logic, and Lemma $2.2$, we can show the following theorem:
\begin{theorem}
Let $Y$ be a separable arc-wise connected metric space in $\mathbb{R}^{n}$, with $f: H \rightarrow Y$ a function of Baire class $\xi + n - 1(H, Y)$ where $\xi \geq 1,$ $E$ a subset of $X$, and $\varphi: E \rightarrow Y$ a boundary function for $f$. Therefore $\varphi$ is of Baire class $\xi + n(E, Y).$
\end{theorem}
Although this does not resolve the fourth open problem of Kaczynski's work, it does provide an extension to a theorem shown, and is valuable nonetheless to the field.

\bibliographystyle{amsplain}

\end{document}